\documentclass[reqno,a4paper,12pt]{amsart} 

\makeatletter
\@ifundefined{subjclassname@2020}{%
  \@namedef{subjclassname@2020}{\textup{2020} Mathematics Subject Classification}%
}{}
\makeatother

\usepackage{amsmath,amscd,amsfonts,amssymb}
\usepackage{mathrsfs,dsfont}
\usepackage{pgfplots}
\pgfplotsset{compat=1.15}
\usepackage{tikz}
\usetikzlibrary{patterns}
\usepackage{float}
\usepackage{appendix}
\usepackage{caption}
\usepackage{subcaption}
\usepackage{enumerate}

\numberwithin{equation}{section}
\numberwithin{figure}{section}

\def\1{\mathds{1}}

\renewcommand\le{\leqslant}
\renewcommand\ge{\geqslant}
\renewcommand\leq{\leqslant}
\renewcommand\geq{\geqslant}

\renewcommand\Im{\operatorname{Im}}

\newcommand{\supp}{\operatorname{supp}}

\newcommand{\diag}{\operatorname{diag}}

\newcommand{\rank}{\operatorname{rank}}

\theoremstyle{plain}
\newtheorem{thm}{Theorem}[section]
\newtheorem{lem}[thm]{Lemma}
\newtheorem{corollary}[thm]{Corollary}

\newtheorem*{claim*}{Claim}
\newtheorem*{thm*}{Theorem}

\newcounter{lettercnt}[section] 

\theoremstyle{definition}
\newtheorem{definition}[thm]{Definition}
\newtheorem*{definition*}{Definition}
\newtheorem*{remarks*}{Remarks}
\newtheorem*{remark*}{Remark}

\newenvironment{enumerate-math}
{\begin{enumerate}
\addtolength{\itemsep}{5pt}
}
{\end{enumerate}}

\newenvironment{enumerate-text}
{\begin{enumerate}
\addtolength{\itemsep}{5pt}
}
{\end{enumerate}}

\begin{document}

\title{Explicit Bounds for Restricted Invertibility with Prescribed Group Cardinalities}

\author{Xinyu Chen}
\email{12432007@mail.sustech.edu.cn}
\author{Bochen Liu}
\email{Bochen.Liu1989@gmail.com}
\address{Department of Mathematics \& International Center for Mathematics, Southern University of Science and Technology, Shenzhen 518055, China}

\thanks{This work was partially supported by the National Key R\&D Program of China 2024YFA1015400. An earlier draft was finished in January 2026. The authors acknowledge the use of generative AI tools for correcting an error from the draft, providing Lemma \ref{barrier-i=j}, suggesting Rado's theorem, and simplifying several arguments. All proofs have been independently verified by the authors and presented in their own words. The authors bear full responsibility for all content of this work. The authors would like to thank Zhiqiang Xu for valuable comments on the manuscript.}
\date{}


\begin{abstract}
We prove an explicit lower spectral bound for independent random
vectors with prescribed group multiplicities. Let
$v_{ij}\in\mathbb C^d$, $1\le i\le l$, $1\le j\le k_i$, be independent
random vectors with finite support and
$\mathbb E v_{ij}v_{ij}^*=A_i$, where
$\sum_{i=1}^l A_i=I_d$ and $\operatorname{Tr}(A_i)\ge\epsilon$.
Set $k=\max_i k_i$, $n=\max_i\operatorname{rank}(A_i)$, and
$K=\sum_{i=1}^l k_i$. If
$n>(6+4\sqrt2)k$ and $2\sqrt{nk}-k<\epsilon\le n/2$, then, with
positive probability,
\[
\lambda_K\left(\sum_{i=1}^l\sum_{j=1}^{k_i}v_{ij}v_{ij}^*\right)
\ge n-(\sqrt{k}+\sqrt{n-\epsilon})^2.
\]
As a consequence, we obtain a restricted invertibility result
selecting exactly $k_i$ vectors from the $i$th group of a Parseval
frame. The proof combines interlacing families, a reciprocal
transformation of mixed characteristic polynomials, and a
multivariate barrier estimate that accommodates repeated
differentiation in the same variable.
\end{abstract}
\maketitle

\section{Introduction}
The problem of extracting a quantitatively linearly independent subfamily from a given collection of vectors arises naturally in functional analysis and harmonic analysis. A fundamental result in this direction is the restricted invertibility principle of Bourgain and Tzafriri 
\cite{BT1987}. It asserts that there exist
absolute constants $c,C>0$ such that, for every matrix
$A\in\mathbb{C}^{n\times n}$ with $\|Ae_j\|_2=1$ for all $j=1,\dots,n$,
one can find a subset $\sigma\subseteq\{1,\ldots,n\}$ satisfying
\begin{equation}
|\sigma|\geq \frac{cn}{\|A\|^2}
\qquad\text{and}\qquad
\left\|\sum_{j\in\sigma}a_jAe_j\right\|_2^2
\geq C\sum_{j\in\sigma}|a_j|^2
\end{equation}
for every choice of scalars $(a_j)_{j\in\sigma}$. Thus, an upper bound on $\|A\|$, together with normalized columns, guarantees a uniform lower bound on a suitably chosen subfamily.

Subsequent work has refined both the quantitative estimates and the methods used to obtain them. In 2012, 
Spielman and Srivastava \cite{SS12} developed an elementary, constructive proof based on a barrier-potential argument, yielding explicit bounds on the smallest singular value in terms of the number of selected columns. In 2022, Marcus, Spielman, and Srivastava \cite{MSS22} approached the problem through interlacing families of polynomials and obtained sharper quantitative estimates. In particular, they proved the following.
\begin{thm}[Theorem 5.1 in \cite{MSS22}]\label{thm-MSS22}
Suppose $u_{1}, \dots, u_{m} \in \mathbb{R}^{d}$ satisfy $\sum_{i \leq m} u_{i} u_{i}^{T}=I_d$ and $k \leq d$ is a positive integer.
Then there exists a subset $S \subset [m]$ of size $k$ such that
\[
\lambda_{k}\left(\sum_{i \in S} u_{i} u_{i}^{T}\right) \geq \frac{\bigl(\sqrt{d(m-k)}-\sqrt{k(m-d)}\bigr)^{2}}{m^{2}},
\]
where $\lambda_k$ denotes the $k$th largest eigenvalue.
\end{thm}

Restricted invertibility is closely related to the Kadison--Singer problem
\cite{CT06,CCLV2005,weaver2004}. The resolution of the latter by
Marcus, Spielman, and Srivastava \cite{MSS15} establishes Weaver’s vector-partitioning conjecture and Anderson’s paving conjecture through interlacing families of polynomials and estimates for mixed characteristic polynomials. We refer to the work of Casazza and Tremain \cite{CT06} for the equivalence between the Kadison--Singer problem and the Bourgain-Tzafriri restricted invertibility conjecture \cite{BT1989}, which asks for a partition with uniform control on every piece, stronger than a single subfamily as in Theorem \ref{thm-MSS22} above. 

The main theorem in \cite{MSS15} is the following.
\begin{thm}\label{MSS-upper}
    If $\varepsilon > 0$ and $v_1, \ldots, v_m$ are independent random vectors in $\mathbb{C}^d$ with finite support such that
$$\sum_{i=1}^m \mathbb{E} \, v_i v_i^* = I_d \quad \text{and} \quad \mathbb{E} \, \|v_i\|^2 \leq \varepsilon, \  \forall\, i=1,\dots,m,$$
then
\[
\mathbb{P}\left( \left\| \sum_{i=1}^m v_i v_i^* \right\| \leq (1 + \sqrt{\varepsilon})^2 \right) > 0.
\]
\end{thm}

Motivated by Theorem \ref{MSS-upper}, we study explicit lower bounds for the smallest eigenvalue of the Gram matrix formed by independent random vectors with prescribed group multiplicities. Compared with the lower bounds obtained by Bownik and Londner
\cite{BL19} and Naor and Youssef \cite{NY17}, our bound is expressed explicitly in terms of covariance traces and ranks and applies to selections with prescribed group multiplicities.

\begin{thm}\label{main-thm}
    Suppose that $v_{ij}$, $1\leq i\leq l$, $1\leq j\leq k_i$ are independent random vectors in $\mathbb C^d$ with finite support and $\mathbb{E}v_{ij}v_{ij}^*=A_i$ for each $i,j$. Assume that
    $$ \sum_{i=1}^{l} A_i=I_d \quad \text{and} \quad \mathrm{Tr}(A_i)=\mathbb E\|v_{ij}\|^2\geq \epsilon>0, \ \forall\, i=1,\dots,l,$$
    with $k:=\max_i k_i$, $n:=\max_i \rank(A_i)$ satisfying
    $$n>(6+4\sqrt{2})k\quad\text{and}\quad 2\sqrt{nk}-k<\epsilon\leq \frac{n}{2}.$$
    Then, with $K:=\sum_{i=1}^l k_i$,
    \[
    \mathbb P\left( \lambda_K\left(\sum_{\substack{1\leq j\leq k_i\\1\leq i\leq l}}v_{ij}v_{ij}^*\right) \geq 
    n-(\sqrt{k}+\sqrt{n-\epsilon})^2 \right)>0,
    \]
    where $\lambda_K(A)$ denotes the $K$-th largest eigenvalue of the matrix $A$. 
\end{thm}
Notice that the lower bound in Theorem \ref{main-thm} is positive under the stated assumptions, so all selected vectors are linearly independent.

As a consequence of Theorem \ref{main-thm}, we obtain a more general restricted invertibility result with prescribed numbers of vectors from each group. More precisely, Theorem \ref{thm-MSS22} selects $k$ vectors from the entire collection, while the following corollary selects exactly $k_i$ vectors from each group $G_i$, $i=1,\ldots,l$.

\begin{corollary}\label{cor}
    Suppose that  \(u_1,\ldots,u_m\in\mathbb C^d\) satisfy $\sum_{i=1}^mu_iu_i^*=I_d$. Partition these vectors into \(l\) groups \(G_1,\ldots,G_l\) with 
    $$N:=\max_i |G_i|,\quad \epsilon:=\min\left\{\min_i \mathrm{Tr}(\sum_{u\in G_i}u u^*), \frac{n}{2}\right\}, \quad n:=\max_i\rank (\sum_{u\in G_i}u u^*). $$
    Then, for all positive integers $k_i, i=1,\dots,l$ satisfying
    $$n>(6+4\sqrt{2})k\quad \text{and}\quad  2\sqrt{nk}-k<\epsilon\leq \frac{n}{2},\ \text{where }k:=\max_i k_i,$$
    there exist $S_i\subseteq G_i$ of size $k_i$, $i=1,\dots,l$, such that, with $K:=\sum_{i=1}^l k_i$,
\[
    \lambda_K\left(\sum_{i=1}^l \sum_{u\in S_i}u u^*\right)\geq \frac{1}{N}\left(n-(\sqrt{k}+\sqrt{n-\epsilon})^2 \right).
\]
\end{corollary}
Notice that the conclusion in Corollary \ref{cor} implies that all selected vectors are distinct, which requires $k_i\leq |G_i|$. Indeed, the conditions in Corollary \ref{cor} imply that $\sum_{u\in G_i} u u^*\leq I_d$ with trace $\geq\epsilon>k\geq k_i$, so $k_i\leq |G_i|$ must hold.

\begin{proof}[Theorem \ref{main-thm} implies Corollary \ref{cor}]
By adding zero vectors, we can enlarge $G_i$ to $\widetilde{G}_i$ such that $|\widetilde{G}_i|=N$ for all $i$. Let $v_i$ be a random vector that takes values in 
 \[
\{\sqrt{N}u : u\in \widetilde G_i\}
 \]
 with equal probability, and let $v_{i,j}, 1\leq j\leq k_i$, be independent copies of $v_i$. Then, with
 $$A_{i}=\mathbb E v_{ij}v_{ij}^*=\sum_{u\in G_i}u u^*,\quad \mathrm{Tr}(A_i)\geq \min_i \mathrm{Tr}(\sum_{u\in G_i}u u^*)\geq \epsilon,$$ 
 and 
 $$\sum_{i=1}^{l}A_i=\sum_{i=1}^{l}\sum_{u\in G_i}u u^*=I_d,$$
 Theorem \ref{main-thm} implies that there exists a choice of $\sum_{i,j}v_{ij}v_{ij}^*$ such that
 \[
 \lambda_K\left(\sum_{i,j}v_{ij}v_{ij}^*\right) \geq 
    n-(\sqrt{k}+\sqrt{n-\epsilon})^2,
 \]
 which means that there exist $S_i\subseteq G_i$ of size $k_i\leq |G_i|$ such that
\[
    \lambda_K\left(\sum_{i=1}^l \sum_{u\in S_i}u u^*\right)\geq \frac{1}{N}\left(n-(\sqrt{k}+\sqrt{n-\epsilon})^2 \right).
\]
 
\end{proof}

Our proof of Theorem \ref{main-thm} uses the flip transformation from \cite{CXX24} to reduce the problem of bounding the smallest eigenvalue from below to bounding the largest eigenvalue from above. We then obtain this upper bound using the multivariate barrier method of \cite{MSS15}. We also derive the associated characteristic polynomial (see Theorem \ref{thm-mix-char-poly} below), which recovers the classical identity \eqref{mix-char-classic} from \cite{MSS15} when all $k_i=1$, and the formula in Lemma 4.2 of \cite{MSS22} when $l=1$ and $A_0=xI_d$.

There is also some works on higher rank extensions of Theorem \ref{MSS-upper}, that replace $v_iv_i^*$by positive semidefinite matrices $X_i$. See, for example,\cite{cohen2016,branden2018,xu2023,bownik2023}. However, the rank-one assumption in Theorem \ref{main-thm}
cannot simply be dropped. Indeed, consider a single deterministic
matrix $X_1=I_{18}$ and take $l=k_1=1$ and $\epsilon=9$.
Then $A_1=\mathbb E X_1=I_{18}$, so $n=18$ and $k=K=1$,
and all the remaining hypotheses of the theorem are satisfied.
Replacing $v_{11}v_{11}^*$ by $X_1$ in its conclusion would give
\[
1=\lambda_1(X_1)
\geq 18-\left(1+\sqrt{18-9}\right)^2=2,
\]
a contradiction.

\subsection*{Organization.}
The paper is organized as follows. Section~2 reviews real stable
polynomials, interlacing families, and barrier functions.
Section~\ref{sec-k=1} proves Theorem~\ref{main-thm} when $k_i=1$
for every $i$. Section~\ref{sec-mixed-poly} derives the mixed
characteristic polynomial identities for general group
multiplicities. Section~\ref{sec-k>1} establishes the barrier
estimate needed for repeated differentiation and completes the
proof of Theorem~\ref{main-thm}. The appendix gives an elementary
proof of the vector-space case of Rado's independent transversal
theorem.

\section{Preliminaries}
In this section, we collect the facts about real stable polynomials, mixed characteristic polynomials, interlacing families, and barrier functions that will be used in the proofs below. We write $[m]=\{1,\ldots,m\}$ and $\partial_i=\partial_{z_i}$.

\subsection{Real stable polynomials}
A polynomial $p\in\mathbb C[z_1,\ldots,z_m]$ is called \emph{stable} if $p\equiv0$ or $p(z_1,\ldots,z_m)\neq0$ whenever $\Im z_i>0$ for every $i$. It is \emph{real stable} if it is stable and has real coefficients. Clearly a nonzero univariate real polynomial is real stable if and only if all its roots are real. The following standard facts can be found in \cite{BB2008}\cite{MSS15}.

\begin{lem}[{\cite[Proposition 2.4]{BB2008}}]\label{realstable-1}
Let $A_1,\ldots,A_m$ be positive semidefinite Hermitian $d\times d$ matrices and let $B$ be Hermitian. Then $\det(B+\sum_{i=1}^m z_iA_i)$ is real stable.
\end{lem}

\begin{lem}[{\cite[Corollary 3.8 and Proposition 3.9]{MSS15}}]\label{realstable-3}
If $p\in\mathbb R[z_1,\ldots,z_m]$ is real stable, then so are $(1-\partial_i)p$ and $p|_{z_i=a}$ for every $i\in[m]$ and $a\in\mathbb R$. Identifying variables also preserves real stability. In particular, if $p\not\equiv0$, then $x\mapsto p(x,\ldots,x)$ is real-rooted.
\end{lem}

\subsection{Mixed characteristic polynomials}
\begin{lem}[{\cite[Theorem 4.1 and Corollary 4.4]{MSS15}}]
Let $v_1,\ldots,v_m$ be independent random vectors in $\mathbb C^d$ with finite support, and set $A_i=\mathbb E\,v_iv_i^*$. Then
\begin{equation}\label{mix-char-classic}
\mathbb E\det\left(xI_d-\sum_{i=1}^m v_iv_i^*\right)
=\left.\left(\prod_{i=1}^m(1-\partial_i)\right)
\det\left(xI_d+\sum_{i=1}^m z_iA_i\right)\right|_{z_1=\cdots=z_m=0}.
\end{equation}
\end{lem}
This polynomial, denoted by $\mu[A_1,\ldots,A_m](x)$, is called the \emph{mixed characteristic polynomial}. It is monic of degree $d$, and all its roots are nonnegative real numbers because, for $x<0$, every realization satisfies
\[
(-1)^d\det\left(xI_d-\sum_{i=1}^m v_iv_i^*\right)>0.
\]

If, in addition, $\sum_{i=1}^m A_i=I_d$, then by translating each $z_i$ by $x$ we have
\begin{equation}\label{translate-x}
\mu[A_1,\ldots,A_m](x)
=\left.\left(\prod_{i=1}^m(1-\partial_i)\right)
\det\left(\sum_{i=1}^m z_iA_i\right)\right|_{z_1=\cdots=z_m=x}.
\end{equation}

\subsection{Interlacing families}
We next recall interlacing families; see \cite{interlacing1,MSS15,MSS22}. For a real-rooted polynomial $f$, let $\lambda_j(f)$ denote its $j$th largest root, counted with multiplicity.
\begin{definition}
A real-rooted polynomial $g$ of degree $d-1$ \emph{interlaces} a real-rooted polynomial $f$ of degree $d$ if
\[
\lambda_1(f)\geq\lambda_1(g)\geq\lambda_2(f)\geq\cdots\geq\lambda_{d-1}(g)\geq\lambda_d(f).
\]
Polynomials $f_1,\ldots,f_r$ have a \emph{common interlacing} if some polynomial interlaces each $f_i$.
\end{definition}

\begin{definition}
Let $S_1,\ldots,S_m$ be finite nonempty sets. For each assignment $(s_1,\ldots,s_m)\in S_1\times\cdots\times S_m$, let $f_{s_1,\ldots,s_m}$ be a real-rooted polynomial of degree $d$ with positive leading coefficient. Define the polynomials for partial assignments recursively by
\[
f_{s_1,\ldots,s_k}:=\sum_{t\in S_{k+1}}f_{s_1,\ldots,s_k,t},
\qquad k=m-1,\ldots,0,
\]
where $k=0$ corresponds to $f_\emptyset$. The polynomials form an \emph{interlacing family} if, for every partial assignment of length $k<m$, the polynomials $\{f_{s_1,\ldots,s_k,t}\}_{t\in S_{k+1}}$ have a common interlacing.
\end{definition}

\begin{lem}[{\cite[Theorem 2.7]{MSS22}}]\label{interlacing-root}
Let $\{f_s\}_{s\in S_1\times\cdots\times S_m}$ be an interlacing family of degree-$d$ polynomials. For each $j\in[d]$, there are assignments $s$ and $\widetilde s$ such that
\[
\lambda_j(f_s)\geq\lambda_j(f_\emptyset)\geq\lambda_j(f_{\widetilde s}).
\]
\end{lem}

\begin{lem}[{\cite[Theorem 4.5]{MSS15}}]\label{interlacing-poly}
Let $v_1,\ldots,v_m$ be independent random vectors in $\mathbb C^d$, with $v_i$ taking values $w_{i,1},\ldots,w_{i,l_i}$ with probabilities $p_{i,1},\ldots,p_{i,l_i}>0$. For $j_i\in[l_i]$, define
\[
q_{j_1,\ldots,j_m}(x)
:=\left(\prod_{i=1}^m p_{i,j_i}\right)
\det\left(xI_d-\sum_{i=1}^m w_{i,j_i}w_{i,j_i}^*\right).
\]
Then $\{q_{j_1,\ldots,j_m}\}$ is an interlacing family, with $q_\emptyset(x)=\mu[A_1,\ldots,A_m](x)$, where $A_i=\mathbb E\,v_iv_i^*$.
\end{lem}

\subsection{Barrier functions}
For a real polynomial $p\in\mathbb R[z_1,\ldots,z_m]$, define the set of points \emph{above the roots} by
\[
\mathrm{Ab}_p:=\{\mathbf z\in\mathbb R^m:p(\mathbf z+\mathbf t)>0
\text{ for every }\mathbf t\in[0,\infty)^m\}.
\]
For real stable $p$ and $\mathbf z\in\mathrm{Ab}_p$, the \emph{barrier function} in direction $i$ is
\begin{equation}\label{root-representation}
\Phi_p^i(\mathbf z):=\frac{\partial_i p(\mathbf z)}{p(\mathbf z)}
=\sum_{j=1}^r\frac{1}{z_i-\lambda_j},
\end{equation}
where $\lambda_1,\ldots,\lambda_r$ are the roots, counted with multiplicity, of the univariate polynomial obtained by fixing all coordinates except the $i$th. By Lemma \ref{realstable-3} and the definition of $\mathrm{Ab}_p$, these roots are real and satisfy $\lambda_j<z_i$. The barrier functions have the following properties.

\begin{lem}[{\cite[Lemma 5.8]{MSS15}}]\label{mono-convex}
If $p\in\mathbb R[z_1,\ldots,z_m]$ is real stable and $\mathbf z\in\mathrm{Ab}_p$, then, for all $i,j\in[m]$,
\[
\Phi_p^i(\mathbf z)\geq0,\qquad
\partial_j\Phi_p^i(\mathbf z)\leq0,\qquad
\partial_{jj}\Phi_p^i(\mathbf z)\geq0.
\]
\end{lem}

\section{The proof of Theorem \ref{main-thm} for $k=1$}\label{sec-k=1}

We first prove Theorem \ref{main-thm} in the case $k=1$. Set $m=l$, write $v_i=v_{i1}$, and let $n_i=\rank(A_i)$, so that $n=\max_i n_i$. Our goal is to find a realization such that
\[
\lambda_m\left(\sum_{i=1}^m v_iv_i^*\right)
\geq n-\left(1+\sqrt{n-\epsilon}\right)^2.
\]
By Lemmas \ref{interlacing-poly} and \ref{interlacing-root}, it suffices to establish the same lower bound for the $m$th largest root of the classical mixed characteristic polynomial $\mu[A_1,\ldots,A_m]$ in \eqref{mix-char-classic}.

Define
\[
P_0(z_1,\ldots,z_m):=\det\left(\sum_{i=1}^m z_iA_i\right),
\]
and, for $1\leq k\leq m$, set
\begin{equation}\label{classical-P-k}
P_k:=(1-\partial_k)P_{k-1}
=\left(\prod_{i=1}^k(1-\partial_i)\right)P_0.
\end{equation}
Each $P_k$ is real stable by Lemmas \ref{realstable-1} and \ref{realstable-3}. Since $\sum_{i=1}^m A_i=I_d$, equation \eqref{translate-x} gives
\begin{equation}\label{def-p-m}
p_m(t):=P_m(t,\ldots,t)=\mu[A_1,\ldots,A_m](t).
\end{equation}

\begin{lem}\label{degree-kone}
For every $0\leq k\leq m$, we have $\deg P_k=d$ and $\deg_{z_i}P_k=n_i$ for each $i\in[m]$.
\end{lem}
\begin{proof}
By its determinant definition, $P_0$ is homogeneous, and $P_0(t,\ldots,t)=t^d$ shows that its degree is $d$. Since differentiation makes the total degree lower, $P_0$ is the highest-degree homogeneous part of every $P_k$.

Fix $i\in[m]$. Choose a unitary matrix $U$ such that
\[
UA_iU^*=\diag(\lambda_1,\ldots,\lambda_{n_i},0,\ldots,0),
\qquad \lambda_1,\ldots,\lambda_{n_i}>0.
\]
Then $\deg_{z_i}P_0\leq n_i$, and the coefficient of $z_i^{n_i}$ is
\[
\lambda_1\cdots\lambda_{n_i}\det M_i,
\]
where $M_i$ is the lower-right $(d-n_i)\times(d-n_i)$ block of $\sum_{h\neq i}z_hUA_hU^*$. Setting $z_h=1$ for every $h\neq i$ gives $M_i=I_{d-n_i}$, so this coefficient is nonzero. Thus $\deg_{z_i}P_0=n_i$.

For $\deg_{z_i} P_k$ with $k>0$, notice that the degrees are preserved inductively under $P_k=(1-\partial_k)P_{k-1}$. More precisely, if $i=k$, differentiation lowers the degree in $z_i$, leaving its leading coefficient unchanged; if $i\neq k$, the leading coefficient in $z_i$ is transformed by $1-\partial_k$, which maps every nonzero polynomial to a nonzero polynomial.
\end{proof}

We apply the flip transformation from \cite{CXX24} by defining
\begin{equation}\label{def-Q-k--}
Q_k^-(z_1,\ldots,z_m):=(-1)^k(z_1\cdots z_m)^n
P_k(z_1^{-1},\ldots,z_m^{-1}).
\end{equation}
Since $\deg_{z_i}P_k\leq n$, this is a polynomial of degree at most $n$ in each variable. If $\Im z_i>0$, then $z_i^{-1}$ lies in the lower half-plane. So real stability of $P_k$ and conjugate symmetry imply that $Q_k^-$ is also real stable. Moreover,
\[
Q_m^-(t,\ldots,t)=(-1)^m t^{nm}p_m(t^{-1}),
\]
where $p_m$ is defined in \eqref{def-p-m}. Since all roots of $p_m$ are nonnegative, the nonzero roots of
$Q_m^-(t,\ldots,t)$ are precisely the reciprocals of the positive
roots of $p_m$. Thus, an upper bound for the largest root of $Q_m^-(t,\ldots,t)$ gives a lower bound for the smallest positive root of $p_m$.

Differentiating the definition of $Q_{k-1}^-$ in \eqref{def-Q-k--} and using \eqref{classical-P-k} gives
\begin{equation}\label{eq-flip-kone}
Q_k^-=(nz_k-1-z_k^2\partial_k)Q_{k-1}^-.
\end{equation}
For a real stable polynomial $p$ with $\deg_{z_k}p\leq n$, define the auxiliary function
\begin{equation}\label{def-G-p}
G_p^k(\mathbf z):=nz_k-1-z_k^2\Phi_p^k(\mathbf z),
\qquad \mathbf z\in\mathrm{Ab}_p.
\end{equation}

\begin{lem}\label{G-monotone-kone}
The function $G_p^k$ is nondecreasing in each coordinate on $\mathrm{Ab}_p$.
\end{lem}
\begin{proof}
For $i\neq k$, Lemma \ref{mono-convex} gives
\[
\partial_iG_p^k(\mathbf z)=-z_k^2\partial_i\Phi_p^k(\mathbf z)\geq0.
\]
For the $k$th coordinate, fix $\mathbf z\in\mathrm{Ab}_p$ and consider the univariate polynomial
\[
q(t):=p(z_1,\ldots,z_{k-1},t,z_{k+1},\ldots,z_m).
\]
Let $\lambda_1,\ldots,\lambda_r$ be its roots, counted with multiplicity, where
$r=\deg q\leq\deg_{z_k}p\leq n$. Then the root representation of the barrier function \eqref{root-representation} yields
\[
\partial_kG_p^k(\mathbf z)
=n-r+\sum_{a=1}^r\frac{\lambda_a^2}{(z_k-\lambda_a)^2}\geq0.
\]
\end{proof}

\begin{proof}[Proof of Theorem \ref{main-thm} for $k=1$]
We seek $x>0$ such that $(x,\ldots,x)\in\mathrm{Ab}_{Q_m^-}$. Let $x>x_0>0$ be chosen later, and set $\delta:=x-x_0$. We will choose $x$ and $x_0$ to satisfy
\begin{equation}\label{eq-step-kone}
\delta\left(\frac{nx-1}{x^2}-\frac{n-\epsilon}{x_0}\right)=1,
\end{equation}
which in particular implies $\frac{nx-1}{x^2}>\frac{n-\epsilon}{x_0}$. For $0\leq j\leq m$, define
\begin{equation}\label{def-a-j}
a_j:=(\underbrace{x,\ldots,x}_{j},
\underbrace{x_0,\ldots,x_0}_{m-j}).
\end{equation}

We now prove by induction that
\begin{equation}\label{eq-invariant-kone}
a_j\in\mathrm{Ab}_{Q_j^-}
\quad\text{and}\quad \Phi_{Q_j^-}^i(a_j)\leq\frac{n-\epsilon}{x_0}\quad\text{for every }i>j.
\end{equation}

For $j=0$, the identity
\[
Q_0^-(\mathbf z)=(z_1\cdots z_m)^n
\det\left(\sum_{i=1}^m\frac{A_i}{z_i}\right)
\]
and $\sum_iA_i=I_d$ show that $Q_0^-(\mathbf z)>0$ whenever every $z_i>0$. Hence $a_0\in\mathrm{Ab}_{Q_0^-}$, and direct differentiation gives
\[
\Phi_{Q_0^-}^i(a_0)
=\frac{n-\operatorname{Tr}(A_i)}{x_0}\leq\frac{n-\epsilon}{x_0},
\qquad i\in[m].
\]

Suppose \eqref{eq-invariant-kone} holds at step $j-1$, and write $p=Q_{j-1}^-$. Since $a_j=a_{j-1}+\delta e_j$, we have $a_j\in\mathrm{Ab}_p$. By monotonicity of the barrier functions (recall Lemma \ref{mono-convex}),
\[
\Phi_p^j(a_j)\leq\Phi_p^j(a_{j-1})\leq\frac{n-\epsilon}{x_0},
\]
so the auxiliary function $G_p^j$ defined above satisfies
\[
G_p^j(a_j)=x^2\left(\frac{nx-1}{x^2}-\Phi_p^j(a_j)\right)
\geq x^2\left(\frac{nx-1}{x^2}-\frac{n-\epsilon}{x_0}\right)>0.
\]
For every $\mathbf u\in[0,\infty)^m$, Lemma \ref{G-monotone-kone} therefore gives $G_p^j(a_j+\mathbf u)>0$. Hence, by \eqref{eq-flip-kone},
\[
Q_j^-(a_j+\mathbf u)=G_p^j(a_j+\mathbf u)p(a_j+\mathbf u)>0,
\]
which proves $a_j\in\mathrm{Ab}_{Q_j^-}$.

For $i>j$, differentiating \eqref{eq-flip-kone} yields
$$
\Phi_{Q_j^-}^i(a_j)=\Phi_p^i(a_j)
-\frac{\partial_j\Phi_p^i(a_j)}{\frac{nx-1}{x^2}-\Phi_p^j(a_j)}\leq\Phi_p^i(a_j)-\delta\partial_j\Phi_p^i(a_j),$$
where the last inequality follows from $\partial_j\Phi_p^i\leq0$ and
$\frac{nx-1}{x^2}-\Phi_p^j(a_j)
\geq\frac{nx-1}{x^2}-\frac{n-\epsilon}{x_0}=\frac1\delta$.

Finally, by convexity in the $j$th coordinate, we conclude that
$$
\Phi_{Q_j^-}^i(a_j)\leq\Phi_p^i(a_{j-1})\leq\frac{n-\epsilon}{x_0},$$
which completes the proof of \eqref{eq-invariant-kone}.

It remains to choose $x$ and $x_0$ satisfying \eqref{eq-step-kone}. Writing $x=(r+1)x_0$ with $r>0$, that condition becomes
\[
\frac1x=\epsilon-1-r(n-\epsilon)-\frac1r.
\]
The right-hand side is maximized at $r=(n-\epsilon)^{-1/2}$, giving
\[
x=\frac1{\epsilon-1-2\sqrt{n-\epsilon}}
=\frac1{n-(1+\sqrt{n-\epsilon})^2},
\qquad
x_0=\frac{x}{1+(n-\epsilon)^{-1/2}}.
\]
The assumptions of Theorem \ref{main-thm} imply $n-\epsilon>0$ and $x>x_0>0$, so these choices are valid. By \eqref{def-a-j} and \eqref{eq-invariant-kone}, we have $(x,\ldots,x)\in\mathrm{Ab}_{Q_m^-}$. Therefore the largest root of $Q_m^-(t,\ldots,t)$ is at most $x$, and the smallest positive root of $p_m$ is at least $1/x$.

We still need to show that $p_m$ defined in \eqref{def-p-m} has $m$ nonzero roots. Let $R$ be the maximum rank of $\sum_{i=1}^m v_iv_i^*$ over all realizations in the support. Then every polynomial $\det\left(yI_d-\sum_{i=1}^m v_iv_i^*\right)$ is divisible by $y^{d-R}$. More precisely, the coefficient of $y^{d-R}$ is zero when $\rank\left(\sum_{i=1}^m v_iv_i^*\right)<R$, and equals
\[
(-1)^R\prod_{a=1}^R\lambda_a\left(\sum_{i=1}^m v_iv_i^*\right)
\]
when $\rank\left(\sum_{i=1}^m v_iv_i^*\right)=R$. These nonzero coefficients all have the same sign, so their positive-probability average cannot vanish. Thus $p_m(y)=\mathbb E\det\left(yI_d-\sum_{i=1}^m v_iv_i^*\right)$ has exactly $d-R$ zero roots. Since all its roots are nonnegative, the number $D$ of positive roots, counted with multiplicity, satisfies $D=R$.

The rank argument establishing $D=K$ at the end of Section \ref{sec-k>1}, applied with $k_i=1$ for every $i$, gives $D=m$. Lemmas \ref{interlacing-poly} and \ref{interlacing-root} now yield a realization such that
\[
\lambda_m\left(\sum_{i=1}^m v_iv_i^*\right)\geq\lambda_m(p_m)\geq\frac1x
=n-(1+\sqrt{n-\epsilon})^2.
\]
\end{proof}


\section{The associated mixed characteristic polynomial for $k\geq 2$}\label{sec-mixed-poly}
To prove the case $k\geq 2$ in Theorem \ref{main-thm}, we need to generalize the classical mixed characteristic polynomial identity \eqref{mix-char-classic} to prescribed numbers of vectors from each group.
\begin{thm}\label{thm-mix-char-poly}
Let $v_{ij}\in\mathbb C^d$, $1\leq i\leq l$, $1\leq j\leq k_i$, be independent random vectors satisfying $\mathbb E\,v_{ij}v_{ij}^*=A_i$ for every $i,j$. Then, for every $A_0\in\mathbb C^{d\times d}$,
\begin{align}
\mathbb{E}\det\left(A_0-\sum_{i,j}v_{ij}v_{ij}^*\right)
&=
\left.
\prod_{i=1}^{l}(1-\partial_{z_i})^{k_i}
\det\left(A_0+\sum_{i=1}^{l}z_iA_i\right)
\right|_{z_1=\cdots=z_l=0}
\label{eq-mixcharacter-2}
\\
&=
\left.
\det\left(A_0-\sum_{i=1}^{l} A_i\partial_{z_i}\right)
\prod_{i=1}^{l} z_i^{k_i}
\right|_{z_1=\cdots=z_l=1}
\label{eq-mixcharacter-1}
\end{align}
\end{thm}

Notice that, with $A_0=xI_d$, identity \eqref{eq-mixcharacter-2} reduces to \eqref{mix-char-classic} when all $k_i=1$; when $l=1$, $k_1=k$, $A_0=xI_d$, we can take $z_1=z/x$ with $\frac{\partial }{\partial_{z_1}}=x\frac{\partial }{\partial_z}$ and reduce \eqref{eq-mixcharacter-1} to
$$\det(xI_d-xA_1\partial_z) (z/x)^k|_{z=x}=x^{d-k}(\det(I_d-A_1\partial_x)x^k),$$
which recovers \cite[Lemma 4.2]{MSS22}.

\begin{proof}[Proof of Theorem \ref{thm-mix-char-poly}]
Set
\[
F(z_1,\ldots,z_l):=\det\left(A_0+\sum_{i=1}^l z_iA_i\right).
\]
For any fixed matrix $B\in\mathbb C^{d\times d}$ and random vector $v\in\mathbb C^d$ with finite second moment, the rank-one determinant identity (see, e.g, \cite[Lemma 4.2]{MSS15}) gives
\[
\mathbb E\det(B-vv^*)
=\left.(1-\partial_t)\det\left(B+t\mathbb E\,vv^*\right)\right|_{t=0}.
\]
Applying this identity successively and using independence, we obtain
$$\mathbb E\det\left(A_0-\sum_{i,j}v_{ij}v_{ij}^*\right)=\left.
\left(\prod_{i=1}^l\prod_{j=1}^{k_i}(1-\partial_{t_{ij}})\right)
F\left(\sum_{j=1}^{k_1}t_{1j},\ldots,\sum_{j=1}^{k_l}t_{lj}\right)
\right|_{t_{ij}=0}.$$
By the chain rule, each $\partial_{t_{ij}}$ acts as $\partial_{z_i}$ on $F$. Hence the right-hand side equals
\[
\left.\prod_{i=1}^l(1-\partial_{z_i})^{k_i}F(z_1,\ldots,z_l)
\right|_{z_1=\cdots=z_l=0},
\]
which proves \eqref{eq-mixcharacter-2}.

To obtain \eqref{eq-mixcharacter-1}, note that for any nonnegative integers $k,s$,
\[
\left.(1-\partial_z)^kz^s\right|_{z=0}
=\left.(-\partial_z)^sz^k\right|_{z=1}
=\begin{cases}
(-1)^s\dfrac{k!}{(k-s)!},&0\leq s\leq k,\\[4pt]
0,&s>k.
\end{cases}
\]
Applying this identity to each variable in each monomial of $F$, and using linearity, gives
$$\left.\prod_{i=1}^l(1-\partial_{z_i})^{k_i}F(z_1,\ldots,z_l)
\right|_{z_1=\cdots=z_l=0}=\left.
F(-\partial_{z_1},\ldots,-\partial_{z_l})
\prod_{i=1}^l z_i^{k_i}
\right|_{z_1=\cdots=z_l=1}.$$
Since
\[
F(-\partial_{z_1},\ldots,-\partial_{z_l})
=\det\left(A_0-\sum_{i=1}^l A_i\partial_{z_i}\right),
\]
this is \eqref{eq-mixcharacter-1}.
\end{proof}

\section{The proof of Theorem \ref{main-thm} for $k\geq 2$}\label{sec-k>1}
Set $K:=\sum_{i=1}^l k_i$. By \eqref{eq-mixcharacter-2} and $\sum_iA_i=I_d$, the expected characteristic polynomial is
\[
\mu(x):=\mathbb E\det\left(xI_d-\sum_{i,j}v_{ij}v_{ij}^*\right)
=\left.\prod_{i=1}^l(1-\partial_i)^{k_i}
\det\left(\sum_{i=1}^l z_iA_i\right)\right|_{z_1=\cdots=z_l=x}.
\]
The roots of $\mu$ are real and nonnegative, as recalled in Section 2. Our goal is to prove that $\lambda_K(\mu)\geq n-(\sqrt{k}+\sqrt{n-\epsilon})^2$, which yields Theorem \ref{main-thm} by Lemmas \ref{interlacing-poly} and \ref{interlacing-root}. Compared with Section \ref{sec-k=1}, the new ingredient is a barrier estimate for repeated application of the operator in the same direction.

For $1\leq i\leq l$ and $0\leq s\leq k_i$, define
\[
P_{i,s}:=\left(\prod_{j=1}^{i-1}(1-\partial_j)^{k_j}\right)
(1-\partial_i)^s\det\left(\sum_{j=1}^l z_jA_j\right).
\]
As explained in the proof of Lemma \ref{degree-kone}, each application of $1-\partial_j$ preserves the total degree and the degree in each variable. Thus every $P_{i,s}$ is real stable, has total degree $d$, and has degree $\rank(A_j)\leq n$ in $z_j$.

Similar to Section \ref{sec-k=1}, we define the real stable polynomials
\[
Q_{i,s}^-(z_1,\ldots,z_l)
:=(-1)^{s+\sum_{j<i}k_j}(z_1\cdots z_l)^n
P_{i,s}(z_1^{-1},\ldots,z_l^{-1}).
\]
Their degree in each variable is at most $n$, and
\begin{equation}\label{eq-flip-general}
Q_{i,s}^-=(nz_i-1-z_i^2\partial_i)Q_{i,s-1}^-,
\qquad 1\leq s\leq k_i.
\end{equation}
Moreover, $Q_{i,0}^-=Q_{i-1,k_{i-1}}^-$ for $i\geq2$. We use the auxiliary functions $G_p^i$ and their monotonicity from Lemma \ref{G-monotone-kone}, together with the following estimate.

\begin{lem}\label{barrier-i=j}
Let $p\in\mathbb R[z_1,\ldots,z_l]$ be real stable with $\deg_{z_i}p\leq n$, and assume that $(nz_i-1-z_i^2\partial_i)p$ is also real stable. Suppose that $\mathbf a\in\mathrm{Ab}_p\cap(0,\infty)^l$, $\beta,\delta>0$, and $\Phi_p^i(\mathbf a)\leq\beta$. Set $\mathbf b=\mathbf a+\delta e_i$. If
\[
\beta<\frac{nb_i-1}{b_i^2},\qquad
\delta\geq\frac1{\frac{nb_i-1}{b_i^2}-\beta},\qquad
a_i\beta\geq\frac n2,
\]
then
\[
\mathbf b\in\mathrm{Ab}_{(nz_i-1-z_i^2\partial_i)p},\qquad
\Phi_{(nz_i-1-z_i^2\partial_i)p}^i(\mathbf b)\leq\beta.
\]
\end{lem}
\begin{proof}
Since $\mathbf b\in\mathrm{Ab}_p$ and $\Phi_p^i(\mathbf b)\leq\Phi_p^i(\mathbf a)\leq\beta$, we have $G_p^i(\mathbf b)>0$. By Lemma \ref{G-monotone-kone}, for every $\mathbf u\geq0$,
\[
\bigl((nz_i-1-z_i^2\partial_i)p\bigr)(\mathbf b+\mathbf u)=G_p^i(\mathbf b+\mathbf u)p(\mathbf b+\mathbf u)>0.
\]
Hence $\mathbf b\in\mathrm{Ab}_{(nz_i-1-z_i^2\partial_i)p}$.

Let $\lambda_1,\ldots,\lambda_r$ be the roots with multiplicity of
$$p_i(y):=p(a_1,\ldots,a_{i-1},y,a_{i+1},\ldots,a_l).$$
Notice that $r=\deg p_i\leq\deg_{z_i}p\leq n$. Set $u_j=(a_i-\lambda_j)^{-1}$ and $\nu_j=(b_i-\lambda_j)^{-1}$. These numbers are positive and satisfy $\nu_j^2/u_j=\nu_j-\delta\nu_j^2$. Then Cauchy--Schwarz gives
$$
\bigl(\Phi_p^i(\mathbf b)\bigr)^2\leq\left(\sum_{j=1}^r u_j\right)
\left(\sum_{j=1}^r\frac{\nu_j^2}{u_j}\right)\leq\beta\bigl(\Phi_p^i(\mathbf b)+\delta\partial_i\Phi_p^i(\mathbf b)\bigr).$$
Recall $\Phi_p^i(\mathbf b)\leq\beta$ from the first line of the proof. So the preceding inequality gives
\begin{equation}\label{ineq-partial}
-\partial_i\Phi_p^i(\mathbf b)
\leq\frac{\Phi_p^i(\mathbf b)\bigl(\beta-\Phi_p^i(\mathbf b)\bigr)}{\delta\beta}.
\end{equation}
The assumption on $\delta$ also yields
\begin{equation}\label{ineq-barrier}
\frac{nb_i-1}{b_i^2}-\Phi_p^i(\mathbf b)
\geq\frac1\delta+\beta-\Phi_p^i(\mathbf b).
\end{equation}
Now we differentiate $(nz_i-1-z_i^2\partial_i)p=G_p^i\cdot p$ (recall \eqref{def-G-p}) to obtain
$$\Phi_{(nz_i-1-z_i^2\partial_i)p}^i(\mathbf b)=\Phi_p^i(\mathbf b)
+\frac{n-2b_i\Phi_p^i(\mathbf b)-b_i^2\partial_i\Phi_p^i(\mathbf b)}{G_p^i(\mathbf b)}.$$
Then, using \eqref{ineq-partial} and \eqref{ineq-barrier}, we conclude that
\[
\begin{aligned}
&G_p^i(\mathbf b)\bigl(\beta-\Phi_{(nz_i-1-z_i^2\partial_i)p}^i(\mathbf b)\bigr)\\
&\qquad\geq b_i^2\bigl(\beta-\Phi_p^i(\mathbf b)\bigr)^2
\left(1+\frac1{\delta\beta}\right)
-2b_i\bigl(\beta-\Phi_p^i(\mathbf b)\bigr)+2b_i\beta-n\\
&\qquad\geq 2a_i\beta-n+2\delta\beta-\frac{\delta\beta}{1+\delta\beta}
\geq0,
\end{aligned}
\]
where the second inequality follows by completing the square, and the last uses $a_i\beta\geq n/2$. Since $G_p^i(\mathbf b)>0$, the result follows.
\end{proof}

\begin{proof}[Proof of Theorem \ref{main-thm} for $k\geq2$]
Choose
\[
r:=\sqrt{\frac{k}{n-\epsilon}},\qquad
x_k:=\frac1{\epsilon-k-2\sqrt{k(n-\epsilon)}},\qquad
x_0:=\frac{x_k}{1+r},
\]
and set
\[
\delta:=\frac{x_k-x_0}{k},\qquad x_s:=x_0+s\delta\quad(0\leq s\leq k).
\]
The assumptions of the theorem imply $n-\epsilon>0$ and $\epsilon-k>2\sqrt{k(n-\epsilon)}$, so $0<x_0<\cdots<x_k$. Direct calculation gives
\[
\frac1{x_0}=\epsilon-3k+(3\epsilon-2n-k)r<\epsilon\leq\frac n2.
\]
Thus $x_0>2/n$, and $x\mapsto(nx-1)/x^2$ is decreasing on $[x_0,\infty)$. For every $1\leq s\leq k$, we consequently have
\begin{equation}\label{eq-step-general}
\begin{aligned}
\frac{nx_s-1}{x_s^2}-\frac{n-\epsilon}{x_0}
&\geq\frac{nx_k-1}{x_k^2}-\frac{n-\epsilon}{x_0}
=\frac1\delta>0,\\
x_{s-1}\frac{n-\epsilon}{x_0}
&\geq n-\epsilon\geq\frac n2.
\end{aligned}
\end{equation}

For $1\leq i\leq l$ and $0\leq s\leq k_i$, set
\[
a_{i,s}:=(x_{k_1},\ldots,x_{k_{i-1}},x_s,
\underbrace{x_0,\ldots,x_0}_{l-i}).
\]
Then we prove inductively that
\begin{equation}\label{eq-invariant-general}
a_{i,s}\in\mathrm{Ab}_{Q_{i,s}^-},\qquad
\Phi_{Q_{i,s}^-}^j(a_{i,s})\leq\frac{n-\epsilon}{x_0}
\quad(j\in[l]).
\end{equation}
At $a_{1,0}=(x_0,\ldots,x_0)$, the same argument as in Section \ref{sec-k=1} gives $a_{1,0}\in \mathrm{Ab}_{Q_{1,0}^-}$ and
\[
\Phi_{Q_{1,0}^-}^j(a_{1,0})
=\frac{n-\operatorname{Tr}(A_j)}{x_0}\leq\frac{n-\epsilon}{x_0}.
\]
Suppose \eqref{eq-invariant-general} holds at $(i,s-1)$, and write
\[
p=Q_{i,s-1}^-,\qquad a=a_{i,s-1},\qquad b=a_{i,s}=a+\delta e_i.
\]

By \eqref{eq-step-general}, Lemma \ref{barrier-i=j} applies with $\beta=(n-\epsilon)/x_0$. It gives $b\in\mathrm{Ab}_{Q_{i,s}^-}$ and establishes
\eqref{eq-invariant-general} for $j=i$. 

For $j\neq i$, differentiating \eqref{eq-flip-general} gives
$$\Phi_{Q_{i,s}^-}^j(b)=\Phi_p^j(b)-\frac{\partial_i\Phi_p^j(b)}{\frac{nx_s-1}{x_s^2}-\Phi_p^i(b)}\leq\Phi_p^j(b)-\delta\partial_i\Phi_p^j(b)\leq\Phi_p^j(a)\leq\frac{n-\epsilon}{x_0},$$
where the first inequality uses \eqref{eq-step-general} and monotonicity of the barrier functions, and the second uses convexity in direction $i$. This proves the induction step. Since $a_{i,k_i}=a_{i+1,0}$ and $Q_{i,k_i}^-=Q_{i+1,0}^-$ for $i<l$, the induction covers all groups.

As a consequence of \eqref{eq-invariant-general}, we have $(x_{k_1},\ldots,x_{k_l})\in\mathrm{Ab}_{Q_{l,k_l}^-}$. As $x_k\geq x_{k_i}$ for every $i$, it follows that
$$(x_k,\ldots,x_k)\in\mathrm{Ab}_{Q_{l,k_l}^-}.$$
Finally the identity
\[
Q_{l,k_l}^-(x,\ldots,x)=(-1)^Kx^{nl}\mu(x^{-1})
\]
implies that the smallest positive root of $\mu$ is at least $1/x_k$.

It remains to show that $\mu$ has exactly $K=\sum k_i$ positive roots, counted with multiplicity. Let $D$ denote this number. The coefficient argument in Section \ref{sec-k=1} shows that $D$ equals the maximum rank of $\sum_{i,j}v_{ij}v_{ij}^*$ over all realizations in the support. In particular, $D\leq K$.

For the reverse inequality $D\geq K$, we use the special case of Rado's independent transversal theorem \cite[Theorem 1]{Rado42} that is originally stated on matroid. For readers' sake, we give a proof of this special case in the Appendix. See Lemma \ref{independent-transversal} below.

To apply Lemma \ref{independent-transversal}, note that the identity $u^*A_i u=\mathbb E|u^*v_{ij}|^2$ implies
\[
\operatorname{span}\supp(v_{ij})=\operatorname{Range}(A_i).
\]
For any subset $J\subseteq\{(i,j):1\leq i\leq l,\ 1\leq j\leq k_i\}$, set
\[
I:=\{i:(i,j)\in J\text{ for some }j\}.
\]
Since $0\leq \sum_{i\in I}A_i\leq I_d$ and $\epsilon>2\sqrt{nk}-k>k$, we have
$$\dim\operatorname{span}\left(\bigcup_{(i,j)\in J}\supp(v_{ij})\right)=\rank \left(\sum_{i\in I}A_i\right)\geq\operatorname{Tr}\left(\sum_{i\in I}A_i\right)\geq |I|\epsilon\geq\sum_{i\in I}k_i\geq |J|.$$
Thus Lemma \ref{independent-transversal} provides linearly independent choices from all the supports $\supp(v_{ij})$. By independence and finite support, $\rank\left(\sum_{i,j}v_{ij}v_{ij}^*\right)=K$ with positive probability, so $D=K$.

Finally, Lemmas \ref{interlacing-poly} and \ref{interlacing-root} yield a realization for which
\[
\lambda_K\left(\sum_{i,j}v_{ij}v_{ij}^*\right)\geq\lambda_K(\mu)\geq\frac1{x_k}
=n-\left(\sqrt{k}+\sqrt{n-\epsilon}\right)^2,
\]
that completes the proof of Theorem \ref{main-thm}.
\end{proof}

\appendix
\section{Rado's independent transversal theorem}
\begin{lem}[a special case of Rado's Theorem from \cite{Rado42}]\label{independent-transversal}
Let \(V\) be a finite-dimensional vector space over $\mathbb{C}$, and let
\[
S_1,\ldots,S_m\subseteq V
\]
be sets of vectors. Suppose that
\begin{equation}
\dim\operatorname{span}\left(\bigcup_{a\in H}S_a\right)
\ge |H|
\qquad\text{for every }H\subseteq\{1,\ldots,m\}.
\label{eq:dimension-condition}
\end{equation}
Then there exist vectors \(w_a\in S_a\), \(1\le a\le m\),
such that \(w_1,\ldots,w_m\) are linearly independent.
\end{lem}

\begin{proof}
We argue by induction on $m$. The case $m=1$ is trivial.

For $m\geq 2$, suppose first that for some nonempty proper subset $H\subsetneq[m]$, the space $W:=\operatorname{span}(\bigcup_{a\in H}S_a)$ has dimension $|H|$. By induction, we can choose $w_a\in S_a$, $a\in H$, forming a basis of $W$. Let $\pi:V\to V/W$ be the quotient map. Then, for every $J\subseteq[m]\setminus H$,
\[
\dim\operatorname{span}\left(\bigcup_{a\in J}\pi(S_a)\right)
=\dim\operatorname{span}\left(\bigcup_{a\in J\cup H}S_a\right)-|H|
\geq |J|.
\]
By the induction hypothesis in $V/W$, we can choose $w_a\in S_a$, $a\notin H$, whose images are linearly independent. Together with the chosen basis of $W$, these vectors form the required linearly independent family.

If, otherwise, i.e.
$$\dim\operatorname{span}\left(\bigcup_{a\in J}S_a\right)\geq |J|+1$$
for every nonempty proper subset $J\subsetneq[m]$. we choose an arbitrary nonzero $w_m\in S_m$, and let $\pi:V\to V/\operatorname{span}\{w_m\}$ be the quotient map. Then, for every nonempty $J\subseteq[m-1]$,
\[
\dim\operatorname{span}\left(\bigcup_{a\in J}\pi(S_a)\right)
\geq\dim\operatorname{span}\left(\bigcup_{a\in J}S_a\right)-1
\geq |J|.
\]
The same condition is automatic for $J=\varnothing$. By induction, we can choose $w_a\in S_a$, $a<m$, whose images are linearly independent. Adjoining $w_m$ gives the required family.
\end{proof}

\bibliographystyle{abbrv}
\bibliography{mybibtex}

\end{document}